\documentclass[12pt]{amsart}
\usepackage[cp1250]{inputenc}
\usepackage[OT4]{fontenc}
\usepackage{amsfonts}
\usepackage{amsmath} 
\usepackage{amssymb}
\usepackage{amsthm}
\usepackage{graphicx}
\usepackage[usenames,dvipsnames,svgnames,table]{xcolor}
\usepackage{enumerate}
\usepackage{mathrsfs}
\usepackage{hyperref}

\newcommand{\mb}[1]{\mathbb{#1}}
\newcommand{\mc}[1]{\mathcal{#1}}

\title[Products of two involutions in $\mc R$]{On strongly reversible elements in the Riordan group}

\author{Roksana S{\l}owik}
\address{Faculty of Applied Mathematics, Silesian University of Technology, Kaszubska~23, 44-100 Gliwice, Poland\\ 
e-mail: roksana.slowik@gmail.com}

\author{Tejbir Lohan}
\address{
Theoretical Statistics and Mathematics Unit,  
Indian Statistical Institute, 
Delhi Center, 
New Delhi 110016, 
India\\
e-mail: tejbirlohan70@gmail.com}

\newtheorem{theorem}{Theorem}[section]

\newtheorem{corollary}{Corollary}[section]
\newtheorem{proposition}{Proposition}[section]

\theoremstyle{definition}

\begin{document}

\begin{abstract} 
	An element of a group is called \emph{reversible} if it is conjugate to its inverse, and \emph{strongly reversible} if it can be expressed as a product of two involutions. 
	We study strongly reversible elements in the Riordan group and its various important subgroups. 
	We prove that not every reversible element in the Riordan group is strongly reversible, and we investigate products of reversible elements in the Riordan group.\\
\textbf{Keywords: }
involution, Riordan array, Riordan group, strongly reversible element, pseudo-involution, commutator.
\\
\textbf{AMS classification: } 15A23, 20H20, 05A15, 20E45.\\  
\end{abstract}

\maketitle

\section{Introduction}

The Riordan group $\mathcal{R}$ plays a significant role in matrix
theory and algebraic combinatorics due to its rich algebraic, analytic,
and combinatorial structure. It was introduced by Shapiro et al.\
\cite{sh_ge_wo_wo} and named after John Riordan. The group consists of
an infinite family of lower triangular matrices with entries from a
commutative ring with identity, known as Riordan matrices
or Riordan arrays. In this paper we assume that this ring is one of
the fields $\mathbb{R}$ or $\mathbb{C}$, which will be universally denoted by $\mathbb{F}$. Each Riordan array
$(g(t), f(t))$ is determined by a pair of formal power series with
$g(0) \neq 0$, $f(0) = 0$, and $f'(0) \neq 0$, where the generating
function of the $k$-th column is $g(t)f(t)^k$. The group operation is
the standard matrix multiplication, as described by the Fundamental
Theorem of Riordan Arrays (FTRA), and corresponds to the multiplication
\[
\left(g(t),\, f(t)\right)\left(u(t),\, v(t)\right)
=
\left(g(t)\,u\!\left(f(t)\right),\, v\!\left(f(t)\right)\right).
\]
This structure enables the systematic study of generating functions,
combinatorial identities, and lattice path enumeration; see
\cite{bar, Shapiro_book_22} for a detailed exposition.

The decomposition of group elements into products of involutions
(i.e., elements of order at most two) is a problem of broad interest
across various areas of mathematics; see~\cite{gu_ha_ra, wo}. An
element of a group is called \emph{reversible} if it is conjugate to
its own inverse, and \emph{strongly reversible} if it is conjugate to
its inverse by an involution. Note that an element is strongly
reversible if and only if it can be expressed as a product of two
involutions. Every strongly reversible element is reversible, although
the converse does not generally hold. For instance, every reversible
element of the general linear group $\mathrm{GL}_n(\mathbb{F})$ over
an arbitrary field $\mathbb{F}$ is strongly reversible
\cite{do, ho_pa, wo}, but this does not hold for the corresponding
special linear group $\mathrm{SL}_n(\mathbb{F})$.

 The
study of reversible and strongly reversible elements in groups forms a
historically rich and active area of research; see the
monograph~\cite{OS} for a comprehensive exposition. In this article,
we investigate this problem in the context of the Riordan
group~$\mathcal{R}$.

The set $\mathfrak{J}$ of involutions in $\mathcal{R}$ and their
properties have been extensively studied in the literature; see
\cite{Ch_ki, Ch_ki_sh, ki, lu_mo_pr_2017, lu_mo_pr_2022, sh} and
references therein. In particular, Luz\'{o}n et al.\ proved that
$\mathfrak{J}$ generates a subgroup $\langle\mathfrak{J}\rangle$ of
$\mathcal{R}$, and every element of $\langle\mathfrak{J}\rangle$ can
be written as a product of at most four Riordan involutions
\cite{lu_mo_pr_2022}. Moreover, this bound is sharp: there exist
elements in $\langle\mathfrak{J}\rangle$ that cannot be expressed as a
product of three or fewer involutions \cite[p.~215]{lu_mo_pr_2022}.
It is known~\cite{ki} that every nonscalar involution\footnote{That is, different from $\pm I$.} in $\mathcal{R}$
is similar to one of the two arrays
\[
M := (1,-t) =
\begin{bmatrix}
	1 & & & & \\
	& -1 & & & \\
	& & 1 & & \\
	& & & -1 & \\
	& & & & \ddots
\end{bmatrix}
\qquad \text{or} \qquad {-M} = (-1,-t).
\]

Note that every involution $R \in \mathcal{R}$ may be treated as a
product of two involutions (as $R = -I \cdot (-R)$), and is hence
trivially strongly reversible. To avoid this ambiguity, throughout
this paper we say that an array is a product of two involutions only
if it is not itself an involution. Recall from~\cite{lu_mo_pr_2023}
that, over a field of characteristic zero, the commutator subgroup
$\mathcal{R}'$ of $\mathcal{R}$ consists of all Riordan arrays with
$1$s on the main diagonal, and every element of $\mathcal{R}'$ is a
single commutator. The following result determines the precise
algebraic position of every product of two nonscalar, that is different from $\pm I$, Riordan
involutions in~$\mathcal{R}$.

\begin{proposition}\label{prop:comm}
	Every product of two Riordan involutions is of the form
	$\pm C$, where $C$ is a commutator in $\mathcal{R}$.
\end{proposition}

Inspired by Proposition~\ref{prop:comm}, we consider in
Section~\ref{subsec:subgroups} several well-studied subgroups of
$\mathcal{R}$, including the derivative, hitting-time, Lagrange, Bell,
reciprocal, and stabilizer subgroups, and show that every product of
two nonscalar involutions is a commutator of two elements \emph{from
	the same subgroup}. These results are summarized in
Table~\ref{tab:sb}.

Recall that a Riordan array $R$ is called a \emph{pseudo-involution}
if $RM$ is an involution~\cite{ca_nk}. In Section~\ref{subsec:pseudo}
we establish the following characterization of strongly reversible
Riordan arrays.

\begin{theorem}\label{thm:str-pseudo}
	A Riordan array $R$ is strongly reversible if and only if it is
	conjugate to a pseudo-involution.
\end{theorem}

Theorem~\ref{thm:str-pseudo} opens a direct route to the study of
strongly reversible elements via the rich existing literature on
pseudo-involutions in $\mathcal{R}$. Combining
Proposition~\ref{prop:comm} with the classification of reversible
elements in $\mathbb{F}[[t]]_1$ - the group of series of the form $\sum_{n=1}^\infty a_nt^n$, $a_1\neq 0$, due to O'Farrell~\cite{ofar} (see
Theorem~\ref{thm:farr}), we obtain the following result, proved in
Section~\ref{sec:rem-rev}.

\begin{theorem}\label{prop:str-rev}
	The set of reversible elements in $\mathcal{R}(\mb C)$ strictly contains the
	set of strongly reversible elements.
\end{theorem}

Theorem~\ref{prop:str-rev} stands in sharp contrast to the
$\mathrm{GL}_n(\mathbb{F})$ setting and demonstrates that the Riordan
group exhibits a genuinely richer structure with respect to
reversibility. Finally, Section~\ref{subsec:prod-rev} discusses
products of two reversible elements in $\mathcal{R}$.
 
\section{Proofs of Results}

\subsection{Products of two involutions in $\mc R$}

\par 
Before starting let us just explain that by the commutator $[A,B]$ we will mean $A^{-1}B^{-1}AB$, whereas by $A^X$ the conjugate $X^{-1}AX$. 
\par 
As it was mentioned earlier, any nonscalar Riordan involution in $\mathcal{R}$ is similar to $\pm M$. Using this fact from \cite{ki}, 
we prove our first result as follows.
\begin{proof}[Proof of 
Proposition	 \ref{prop:comm}] 
	Consider a product of two involutions $I_1=R_1^{-1}MR_1$, $I_2=R_2^{-1}MR_2$. We have 
	$$
	\begin{array}{rl}
		I_1I_2 & =R_1^{-1}MR_1R_2^{-1}MR_2=R_1^{-1}\left(MR_1R_2^{-1}MR_2R_1^{-1}\right)R_1\\
		& \stackrel{M^2=I}{=}R_1^{-1}\left[M^{-1}\left(R_2R_1^{-1}\right)^{-1}MR_2R_1^{-1}\right]R_1
		\\
		& =\left[M,R_2R_1^{-1}\right]^{R_1}=
		\left[M^{R_1},R_1^{-1}R_2\right], 
		\\
	\end{array}
	$$
	i.e. their product is a commutator. 
	Clearly, in the case when both $I_1$, $I_2$ are similar to $-M$, we obtain the same result, whereas for the case when one of them is similar to $-M$, we obtain a negative of a commutator. 
\end{proof}

\par 
Commutator subgroup of the Riordan group was studied in \cite{lu_mo_pr_2023}. Let us recall that from Thm.1 of this paper we learn

\begin{proposition}
	\label{prop:comm_desc}
	If $\mb F$ is a field of characteristic $0$, then the commutator subgroup of the Riordan group $\mc R$ defined over $\mb F$ consists of all Riordan arrays with $1$'s in the main diagonal. 
	Moreover, every element of the commutator subgroup is a (single) commutator. 
\end{proposition}

Thus, every product of two involutions is either contained in the commutator subgroup of $\mc R$ or in the set consisting of negatives of commutators. 

\par 
Observe that in the commutator obtained in the proof of 
Proposition \ref{prop:comm} the array $M^{R_1}$ is an involution. Yet, we cannot expect this from the array $R_1^{-1}R_2$ - from \cite{ofar} it follows that the series conjugating $-t$ with any other involution starts with the term $t$, and thus the arrays $R_1$, $R_2$ have constant diagonals, and so $R_1^{-1}R_2$ must have, so it is not an involution\footnote{Clearly, with exception of the trivial case.}. 
Yet, as promised before, we will see that some subgroups of $\mc R$ possess such a nice property that the products of two involutions from $\mc S$ are commutators of elements of $\mc S$.

\subsection{Strongly reversible elements and pseudo-involutions}
\label{subsec:pseudo}

We now prove Theorem~\ref{thm:str-pseudo}, which establishes a
connection between strongly reversible Riordan arrays and
\emph{pseudo-involutions}.

\begin{proof}[Proof of Theorem	 \ref{thm:str-pseudo}]
	\par 
	Assume some conjugate of $R$ is a pseudo-involution. Then for some $U\in\mc R$ the array $R^U$ is a pseudo-involution and we have $(R^UM)^2=I$. Then also 
	$$
	I=\left((R^UM)^2\right)^{U^{-1}}.
	$$
	Expanding
	$$
	I=\left((R^UM)^2\right)^{U^{-1}}=UU^{-1}RUMU^{-1}RUMU^{-1}=(RUMU^{-1})^2=(RM^{U^{-1}})^2.
	$$
	Denote $M^{U^{-1}}$ by $\tilde M$. Note that, as a conjugate of $M$, the matrix $\tilde M$ is an involution, so $\tilde M^{-1}=\tilde M$. Hence, from $(R\tilde M)^2=I$ we get 
	$$
	R^{-1}=\tilde MR\tilde M=\tilde M^{-1}R\tilde M,
	$$
	that is $R$ is strongly reversible. 
	\par 
	Suppose $R$ is strongly reversible. Then there exists an involution $S$ such that $SRS=S^{-1}RS=R^{-1}$. The latter is equivalent to $(RS)^2=I$, i.e. $RS$ in an involution. 
	Since $S$ is an involution as well, there exists $U\in\mc R$ such that $S^U=\pm M$. 
	We have then $I=\left((RS)^2\right)^U=(R^US^U)^2=(R^UM)^2$, i.e. $R^U$ is a pseudo-involution. 
\end{proof}

\subsection{Remarks on reversible elements} \label{sec:rem-rev}

\par 
So far, we are not able to provide a complete description of reversible elements in the Riordan group. One can notice that since a reversible element $R$ satisfies $X^{-1}RX=R^{-1}$ for some $X$, the main diagonal of $R$ may contain only $1$'s and $-1$'s. Thus, reversible elements can be viewed as products of involutions. Clearly, any involution and a product of two involutions (whose main diagonal is constant) is reversible. Thus, testing whether there exists products of $3$ involutions in $\mc R$ that are reversible would be a vital step in determining reversibility in $\mc R$. Consider then $R\in\mc R$ with the main diagonal of the form $1$, $-1$, $1$, $\ldots$ (The case with the main diagonal $-1$, $1$, $-1$, $\ldots$ may be treated exactly the same way.)
\par 
From the multiplication formula and reversibility definition we obtain 
\begin{corollary}
	If $R=\left(g(t),f(t)\right)\in\mc R$ is reversible in $\mc R$, then $f(t)$ is reversible in 
$\mb F_1[[t]]$ -- the group of power series with composition.  
\end{corollary}

Now, recall the following result from \cite{ofar} that classify reversible elements in 
$\mb F_1[[t]]$. 

\begin{theorem}\cite[Thm.5]{ofar}
	\label{thm:farr}
	Let $f(t)=-t+\cdots$. Then $f(t)$ is reversible if and only if it is conjugate to 
	\begin{equation}
		\label{eq:rever_conj}
		-\frac t{(1+\lambda t^p)^{\frac 1p}}
	\end{equation}
	for some $p\in\mb N$, $\lambda\in \mb C$.
\end{theorem}

\par 
In particular, if $p$ in the above theorem is odd, then the map given by $(\ref{eq:rever_conj})$ is an involution. 
We prove Theorem~\ref{prop:str-rev} by constructing reversible elements of $\mathcal{R}$ that are not strongly reversible.

\begin{proof}[Proof of Theorem \ref{prop:str-rev}]
	Consider the subset $S \subset \mathcal{R}$ defined by
	\[
	S = \left\{ \left(\pm 1,\,-\frac{t}{(1+\lambda t^p)^{\frac{1}{p}}}\right) \colon\; p \in 2\mathbb{N},\ \lambda \in \mathbb{C} \right\}.
	\]
	
	Let $A \in S$. By Theorem~\ref{thm:farr}, $A$ is reversible in $\mathcal{R}$. Clearly, if $\lambda\neq 0$, $A$ is not an involution. Since the main diagonal of $A$ alternates between $1$ and $-1$, Proposition~\ref{prop:comm} implies that $A$ cannot be written as a product of two nonscalar involutions in $\mathcal{R}$. Hence, $A$ is not strongly reversible. This completes the proof.
\end{proof}

\subsection{Products of two reversible elements}\label{subsec:prod-rev}

\par 
Let us finish with an observation on the form of the products of two reversible Riordan arrays. Although we are not able to provide a detailed description of such products, we hope that it might be a starting point for further investigation of the subgroup generated by reversible Riordan arrays. 
\par 
Note that since main diagonal of a triangular matrix does not change after a conjugating by a triangular matrix, main diagonal of any reversible triangular array may consists only of the elements $1$ and $-1$. 
\par 
Clearly all involutions are reversible, as well as the products of two involutions, which we described in the preceding section. Thus, we conclude the following. 

\begin{corollary}
	The set of products of two reversible elements in $\mc R$ contains the subgroup of $\mc R$ which consists of all Riordan arrays whose main diagonals are of the forms $1,1,1,\ldots$, $-1,-1,-1,\ldots$. 
	\par 
	Moreover, if the main diagonal of a Riordan array $R$ is of the form $1,1,1,\ldots$, $-1,-1,-1,\ldots$, then $R$ is reversible. Otherwise, it is a product of at most two reversible arrays. 
\end{corollary}

\section{Products of two involutions in subgroups of $\mc R$}\label{subsec:subgroups}

\par 
The Riordan group $\mathcal{R}$ has several well-studied subgroups defined by constraints on the generating functions $(g(t),f(t))$ (see \cite[Section~3.2]{Shapiro_book_22}). Note that every involution in $\mathcal{R}$ is similar to one of the involutions in the subgroup $\mathcal{K}:=\{\pm I, \pm M\}$ of $\mathcal{R}$, which is isomorphic to $\mathbb{Z}_2 \times \mathbb{Z}_2$. Therefore, to characterize products of two involutions in a subgroup $H$ of $\mathcal{R}$, it suffices to determine whether, for each nonscalar involution in $H$ that is similar to $\pm M$ in $\mathcal{R}$, we can choose a suitable transition (i.e. change-of-basis) matrix in $H$. If this holds, then we can formulate a variation of 
Proposition~\ref{prop:comm} for the subgroup $H$ by using a similar argument as in the proof of 
Proposition~\ref{prop:comm}. 
In what follows, we consider some of these subgroups and apply the above-mentioned strategy to investigate products of two involutions in each case.
\par 
Consider the \textbf{Appell or Toeplitz subgroup} $\mathcal{A} = \{(h(t),t) \in \mathcal{R}: h(0)\neq 0\}$. An element $(h(t),t) \in \mathcal{A}$ is an involution if and only if $(h(t),t)^2 = (1,t)$, i.e., $h(t)^2 = 1$. Thus, $h(t) = \pm 1$, so $\pm I$ are the only involutions in $\mathcal{A}$. 
\par 	
Consider the  subgroup  $B_{c,n}=  \left\{ \left(\frac{1}{1-ct^n},\frac{t}{\sqrt[n]{1-ct^n}}\right)\right\}$. If $\mathrm{char}(F) = 2$, then every element of $B_{c,n}$ is an involution. Yet, as we consider here the fields of characteristic $0$ only the identity element $I$ is an involution.
\par 
Consider now the \textbf{Lagrange or associated subgroup} 
$$\mathcal{L} = \{(1,h(t)) \in \mathcal{R} : h(0)=0,\, h'(0)\neq 0\}.$$ 
The matrix $(1,h(t)) \in \mathcal{L}$ is an involution if and only if $h(t)$ is an involution. Moreover, every nonscalar involution in $\mathcal{L}$ that is similar to $M$ via transition matrix $\left(1,f(t)\right) \in \mathcal{L}$.
\par 
Observe now that the following subgroups 

\begin{itemize}
\item the \textbf{derivative subgroup} 
	$$\mathcal{D} = \{(h'(t),h(t)) \in \mathcal{R} : h(0)=0,\, h'(0)\neq 0\},$$
	
\item the \textbf{hitting-time subgroup} 
	$$\mathcal{H} = \left\{\left(\frac{t h'(t)}{h(t)},h(t)\right) \in \mc{R} : h(0)=0,\, h'(0)\neq 0\right\},$$ 

\item the \textbf{Bell subgroup} 
	$$\mathcal{B} = \{(h(t),t h(t))\in \mathcal{R} : h(0)\neq 0\},$$
	
\item the \textbf{reciprocal subgroup} 
	$$\mathcal{R}_r = \left\{\left(\frac{t^r}{h(t)^r},h(t)\right)\in \mathcal{R} : h(0)=0,\, h'(0)\neq 0,\, r\in\mathbb{N}\right\},$$ 
	
\item the \textbf{the eigenvector or stabilizer subgroup},  which for a fixed formal power series $f(t) \in F[[t]]$ with $f(0) \neq 0$, is defined by  
	$$
	\mathrm{Stab}(f) = \left\{ \left( \frac{f(t)}{f(h(t))},\, h(t) \right) \in \mathcal{R} : h(0)=0,\, h'(0)\neq 0 \right\},
	$$
\end{itemize}

are isomorphic to $\mc L$. As isomorphisms preserve commutators, this yields the same result concering the products of two involutions as for $\mc L$.  

\par 
All the results obtained above are summarized in Table~\ref{tab:sb}. 

\begin{table}
	\centering
	\caption{ 
		Characterization of involutions and products of two nonscalar involutions in various subgroups of $\mathcal{R}$}\label{tab:sb}
	\small
	\begin{tabular}{|c|p{1.7cm}|p{3.5cm}|p{4.2cm}|p{2.8cm}|}
		\hline
		\textbf{No.} & \textbf{Subgroup} & \textbf{Elements} & \textbf{Involutions} & \textbf{Product of two nonscalar involutions} \\
		\hline\hline
		1 & $\mc D$ & $(h'(t),h(t))$ & $(h'(t),h(t))$, where $h(t)$ is an involution & Commutator in $\mc D$\\[0.5ex]
		\hline
		2 & $\mc{H}$ & $\left(\frac{th'(t)}{h(t)},h(t)\right)$ & $\left(\frac{th'(t)}{h(t)},h(t)\right)$, where $h(t)$ is an involution & Commutator in $\mc H$ \\[0.5ex]
		\hline
		3 & $\mc{L}$ & $(1,h(t))$ & $(1,h(t))$, where $h(t)$ is an involution & Commutator in $\mc L$ \\[0.5ex]
		\hline
		4 & $\mc{B}$ & $(h(t),th(t))$ & $(h(t),th(t))$, where $th(t)$ is an involution & Commutator in $\mc B$ \\[0.5ex]
		\hline
		5 & $\mc{R}_r $ & $\left(\frac{t^r}{(h(t))^r},h(t)\right)$ & $\left(\frac{t^r}{(h(t))^r},h(t)\right)$, where $h(t)$ is an involution & Commutator in $\mc R_r$ \\[0.5ex]
		\hline
		6& $\mathrm{Stab}(f)$ & $\left(\frac{f(t)}{f(h(t))},h(t)\right)$ & $\left(\frac{f(t)}{f(h(t))},h(t)\right)$, where $h(t)$ is an involution and $f(t)$ is arbitrary & Commutator in $\mathrm{Stab}(f)$ \\[0.5ex]
		\hline
		7 & $\mathcal{A}$ & $(h(t),t)$ & $\{\pm I\}$ & $\emptyset$ \\[0.5ex]
		\hline
		8 & $B_{c,n}$ & $\left(\frac{1}{1-ct^n},\frac{t}{\sqrt[n]{1-ct^n}}\right)$ & $\{I\}$ & $\emptyset$ \\[0.5ex]
		\hline
	\end{tabular}
\end{table}

\medspace
\medspace

\newpage 
\par 
\textbf{Disclosure statement}
\par 
The authors declare that they have no relevant or material financial interests that relate to the research described in this paper.

\end{document}